\documentclass[12pt, reqno]{amsart}
\usepackage{amssymb, graphicx, color}
\usepackage[bookmarksnumbered, colorlinks, plainpages]{hyperref}
\hypersetup{colorlinks=true,linkcolor=red, anchorcolor=green, citecolor=cyan, urlcolor=red, filecolor=magenta, pdftoolbar=true}
\allowdisplaybreaks

\newtheorem{thm}{Theorem}[section]

\newtheorem{cor}[thm]{Corollary}
\newtheorem{lem}[thm]{Lemma}

\newtheorem{prop}[thm]{Proposition}

\newtheorem{remark}[thm]{Remark}
\newtheorem{rk}[thm]{Remark}

\theoremstyle{definition}
\theoremstyle{remark}
\numberwithin{equation}{section}

\newcommand{\R}{{\mathbb R}}

\newcommand{\Z}{{\mathbb Z}}
\newcommand{\C}{{\mathbb C}}

\newcommand{\M}{{\mathcal M}}

\newcommand{\bs}{\begin{split}}
\newcommand{\es}{\end{split}}

\newcommand{\be}{\begin{eqnarray*}}
\newcommand{\ee}{\end{eqnarray*}}
\newcommand{\beq}{\begin{align}}
\newcommand{\eeq}{\end{align}}

\def\Q{\mathcal{Q}}
\def\1{\mathbf{1}}

\begin{document}

%
%
%
%
%
%
%
%
\setcounter{page}{1}
\title[square function and ergodic theory]
{A noncommutative weak type $(1,1)$ estimate for a square function from ergodic theory}

\author[G. Hong]{Guixiang Hong}
\address{
School of Mathematics and Statistics\\
Wuhan University\\
Wuhan 430072\\
China}

\email{guixiang.hong@whu.edu.cn}
\author[B. Xu]{Bang Xu$^{*}$}
\address{
School of Mathematics and Statistics\\
Wuhan University\\
Wuhan 430072\\
China}

\email{bangxu@whu.edu.cn}

\thanks{This work
 was partially supported by Natural Science Foundation of China (Grant:  11601396)}


\date{May 20, 2020.
\newline \indent $^{*}$Corresponding author}
\begin{abstract}
In this paper, we investigate the boundedness of a square function operator from ergodic theory acting on noncommutative $L_{p}$-spaces. The main result is a weak type $(1,1)$ estimate of this operator. We also show the $(L_{\infty},\mathrm{BMO})$ estimate, and thus all the strong type $(L_{p},L_{p})$ estimates by interpolation. The main new difficulty lies in the fact that the kernel of this square function operator does not enjoy any regularity, while the Lipschitz  regularity assumption is crucial in showing such endpoint estimates for the noncommutative Calder\'on-Zygmund singular integrals.
\end{abstract}

\maketitle

\section{Introduction}
Inspired by quantum mechanics and probability, noncommutative harmonic analysis has become an independent field of mathematical research. By using new functional analytic methods from operator space theory and quantum probability, various problems in noncommutative harmonic analysis have been investigated (see, for instance, \cite{Ha,JMX,JPX,JX1,JX,PX1,Vo,JP1,MP}). Especially, Parcet et al developed a remarkable operator-valued Calder{\'o}n-Zygmund theory. More precisely, Parcet \cite{JP1} formulated a noncommutative version of Calder{\'o}n-Zygmund decomposition using the theory of noncommutative martingales and developed a pseudo-localisation principle for singular integrals which is new even in classical theory (see \cite{TH1} for more results on this principle). As a result, Parcet obtained the weak type $(1,1)$ estimates of  Calder{\'o}n-Zygmund operators acting on operator-valued functions. This result played an important role in the perturbation theory \cite{CPSZ}, where the weak type $(1,1)$ estimates were exploited to solve the Nazarov-Peller conjecture.

Later on, Mei and Parcet \cite{MP} proved a weak type $(1,1)$ estimate for a
large class of noncommutative square functions, see \cite{HLMP} for more related results. However, it seems that Mei and Parcet's weak type estimate could not be used to get $(L_{p},L_{p})$ estimate (for $1<p<2$) by interpolation, since the decomposition does not linearly depend on the original functions. This drawback could be revised through operator-valued Calder{\'o}n-Zygmund theory---Proposition 4.3 in \cite{C} where the author proposed a simplified version of Parcet's arguments \cite{JP1}, together with noncommutative Khintchine's inequality as considered in \cite{Pis09, PiRi17}. Moreover, using Khintchine's inequality for weak $L_1$ space considered in \cite{C1}, one gets another kind of weak type $(1,1)$ inequality for the Calder{\'o}n-Zygmund operators with Hilbert valued kernels acting on operator valued functions.

Note that the arguments in \cite{JP1,MP, C} depend heavily on the Lipschitz's regularity of the underlying kernel. In this paper, motivated by the study of noncommutative maximal inequality, we will establish a weak type $(1,1)$ estimate for a square function from ergodic theory. And this square function is different from the class of Calder{\'o}n-Zygmund operators considered in the previous papers \cite{JP1,MP,C} since the associated kernel does not enjoy any regularity.

\vskip 2mm

To illustrate our motivation and present the main results, we need to set up some notions and notations, and refer the reader to Section 2 for more detailed information. Let $\M$ be a von Neumann algebra equipped with a normal semi-finite
faithful (\emph{n.s.f.}) trace $\tau$ and $\mathcal N=L_{\infty}(\R^{d})\overline{\otimes}\M$ be the von Neumann algebra tensor product with the tensor trace $\varphi=\int_{\mathbb R^d}\otimes \tau$.
Let $f:  \R^d\rightarrow S_{\M}$ be locally integrable, where $S_{\M}$ is the subset of $\M$ with $\tau$-finite support. For $t\in\R$, denote $B_t$ to be the open ball centered at the origin $0$ with radius equal to $2^{-t}$. Then we define the averaging operator on $\R^d$ as
$$M_t f(x)= \frac{1}{|B_t|} \int_{B_t}f(x+y)dy= \frac{1}{|B_t|} \int_{\R^d}f(y) \1_{B_t}(x-y)~dy, \quad x\in\R^d.$$
Given $k\in\Z$, $\mathsf{E}_{k}$ denote the $k$-th conditional expectation associated to the sigma algebra generated by the standard dyadic cubes with side-length equal to $2^{-k}$. The sequence of operators that we are going to investigate in the present paper is defined as follows:
\begin{align}\label{1}
T_{k}f(x)=(M_{k}-\mathsf{E}_{k})f(x).
\end{align}
In the scalar-valued case, that is, replacing $\M$ by the set of complex numbers $\C$, the square function
\begin{align}\label{L func}Lf(x)=\big(\sum_{k}|(M_{k}-\mathsf{E}_{k})f(x)|^{2}\big)^{\frac{1}{2}}
\end{align}
 plays an important role in deducing variational inequalities for ergodic averages or averaging operators from the ones for martingales.

 The variational inequalities are much stronger than  the maximal inequalities and imply pointwise convergence immediately without knowing a priori pointwise convergence on a dense subclass of functions, which are absent in some models of dynamical systems. Let us recall briefly the history of the development of the variational inequalities.
This line of research started with L\'epingle's work \cite{Lep76} on martingales which improved the classical Doob maximal inequality. The first variational inequality for the ergodic averages of a dynamical system proved by Bourgain \cite{Bou89} has opened up a new research direction in ergodic theory and harmonic analysis.  Bourgain's work has been extended to many other kinds of operators in ergodic theory and harmonic analysis. For instance, Campbell et al \cite{CJRW00,CJRW03} first proved the variational inequalities associated with singular integrals. The reader is referred to  \cite{JKRW98,JRW03,LeXu2,JSW08,DMT12,Mas,MaTo12,MaTo,JoWa04,OSTTW12,HM1} and references therein for more information on the development of ergodic theory and harmonic analysis in this direction of research.

The square function \eqref{L func} appeared in most of the above references on variational inequalities, and play an important role. In the present paper, similarly, using the noncommutative square function estimates, we provide another proof of the noncommutative Hardy-Littlewood maximal inequalities (or ergodic maximal inequalities) combined with the noncommutative Doob's maximal inequalities, see Corollary \ref{cor2}.

The statement of our result requires the so-called column and row function spaces \cite{P2}. We refer the reader to Section 2 for definitions of noncommutative $L_p$ spaces and weak $L_1$ space---$L_{1,\infty}$. Let $1\leq p\leq\infty$, and $(f_{k})$ be a finite sequence in $L_{p}(\mathcal N)$. Define
$$\|(f_{k})\|_{L_{p}(\mathcal N; \ell_{2}^{r})}=\|(\sum_{k}|f^{\ast}_{k}|^{2})^{\frac{1}{2}}\|_{p},\ \|(f_{k})\|_{L_{p}(\mathcal N; \ell_{2}^{c})}=\|(\sum_{k}|f_{k}|^{2})^{\frac{1}{2}}\|_{p}.$$
Then define ${L_{p}(\mathcal N; \ell_{2}^{r})}$ (resp. ${L_{p}(\mathcal N; \ell_{2}^{c})}$) to be the completion of all finite sequences in $L_p(\mathcal N)$ with respect to $\|\cdot\|_{{L_{p}(\mathcal N; \ell_{2}^{r})}}$ (resp. $\|\cdot\|_{{L_{p}(\mathcal N; \ell_{2}^{c})}}$). The space $L_p(\mathcal{N};
\ell_{2}^{rc})$ is defined as follows.
\begin{itemize}
\item If $2\leq p\leq\infty$,
$$L_p(\mathcal{N};
\ell_{2}^{rc})=L_{p}(\mathcal N; \ell_{2}^{c})\cap L_{p}(\mathcal N; \ell_{2}^{r})$$
equipped with the intersection norm:
$$\|(f_{k})\|_{L_p(\mathcal{N};
\ell_{2}^{rc})}=\max\{\|(f_{k})\|_{L_{p}(\mathcal N; \ell_{2}^{c})},
\|(f_{k})\|_{L_{p}(\mathcal N; \ell_{2}^{r})}\}.$$
\item If $1\leq p<2$,
$$L_p(\mathcal{N};
\ell_{2}^{rc})=L_{p}(\mathcal N; \ell_{2}^{c})+ L_{p}(\mathcal N; \ell_{2}^{r})$$
equipped with the sum norm:
$$\|(f_{k})\|_{L_p(\mathcal{N};
\ell_{2}^{rc})}=\inf\{\|(g_{k})\|_{L_{p}(\mathcal N; \ell_{2}^{c})}+
\|(h_{k})\|_{L_{p}(\mathcal N; \ell_{2}^{r})}\},$$
where the infimun runs over all decompositions $f_{k}=g_{k}+h_{k}$ with $g_{k}$ and $h_{k}$ in $L_{p}(\mathcal{N})$.
\end{itemize}
It is obvious that
$L_2(\mathcal{N}; \ell_{2}^{r}) = L_2(\mathcal{N};
\ell_{2}^{c})=L_2(\mathcal{N};
\ell_{2}^{rc})$.
This procedure is also used to define the spaces
$L_{1,\infty}(\mathcal{N}; \ell_{2}^{r})$ (resp.
$L_{1,\infty} (\mathcal{N}; \ell_{2}^{c})$) and  $L_{1,\infty} (\mathcal{N}; \ell_{2}^{rc})$ with the sum norm,
$$\|(f_{k})\|_{L_{1,\infty}(\mathcal{N};
\ell_{2}^{rc})}=\inf_{f_{k}=g_{k}+h_{k}}\big\{\|(g_{k})\|_{L_{1,\infty}(\mathcal{N}; \ell_{2}^{c})}+\|(h_{k})\|_{L_{1,\infty}(\mathcal{N}; \ell_{2}^{r})}\big\}.$$

We also recall the definitions of $\mathrm{BMO}$ spaces associated
to the von Neumann algebra tensor product $\mathcal A=\mathcal N
\bar\otimes \mathcal{B}(\ell_{2})$ with the tensor trace $\psi=\varphi\otimes tr$ where $tr$ is the canonical trace on $\mathcal{B}(\ell_{2})$. Let $L_0(\mathcal A)$ stand for the $\ast$-algebra of $\psi$-measurable operators affiliated with $\mathcal A$.
According to \cite{MP}, we define the dyadic $\mathrm{BMO}$ space $\mathrm{BMO}_{d}(\mathcal A)$ as the subspace of $L_0(\mathcal A)$ such that
$$\|f\|_{\mathrm{BMO}_{d}(\mathcal A)} \, = \, \max \Big\{
\|f\|_{\mathrm{BMO}_{\! d}^r(\mathcal A)}, \|f\|_{\mathrm{BMO}_{\! d}^c(\mathcal A)}
\Big\} \, < \, \infty,$$ where the row and column dyadic $\mathrm{BMO}_{d}$
norms are given by
\begin{eqnarray*}
\|f\|_{\mathrm{BMO}_{\! d}^r(\mathcal A)} & = & \sup_{Q \in\Q} \Big\| \Big(
\frac{1}{|Q|} \int_Q \Big|\big(f(x) -\frac1{|Q|}\int_Q f(y)dy\big)^{\ast}\Big|^2 \, dx \Big)^{\frac12} \Big\|_{\M\bar\otimes \mathcal{B}(\ell_{2})}, \\
\|f\|_{\mathrm{BMO}_{\! d}^c(\mathcal A)} & = & \sup_{Q \in\Q} \Big\| \Big(
\frac{1}{|Q|} \int_Q \Big|f(x) -\frac1{|Q|}\int_Q f(y)dy\Big|^2 dx \Big)^{\frac12} \Big\|_{\M\bar\otimes \mathcal{B}(\ell_{2})}.
\end{eqnarray*}
We refer the reader to \cite{M} for more precise definitions and relative properties of $\mathrm{BMO}_{d}(\mathcal A)$.

Let $T_{k}$ ($k\in\mathbb Z$) be defined as in (\ref{1}).
The following is our main result.
\begin{thm}\label{t5}
Let $1\leq p\leq \infty$. Then the following assertions are true with a positive constant $C_{p,d}$ depending only on $p$ and the dimension
$d$:
\begin{itemize}
\item[(i)] for $p=1$, $$\|(T_{k}f)\|_{L_{1,\infty}(\mathcal{N},
\ell_{2}^{rc})}\leq C_{p,d}\|f\|_{1},\; \forall f\in L_{1}(\mathcal N);$$

\item[(ii)] for $p=\infty$, $$\Big\|
\sum_{k} T_k \hskip-1pt f \otimes e_{1k}
\Big\|_{\mathrm{BMO}_{d}(\mathcal{A})} +
\Big\|
\sum_{k} T_k \hskip-1pt f \otimes e_{k1}
\Big\|_{\mathrm{BMO}_{d}(\mathcal{A})} \leq C_{p,d} \,
\|f\|_\infty,\; \forall f\in L_{\infty}(\mathcal N);$$

\item[(iii)] for $1<p<\infty$, $$\| (T_{k}f)\|_{L_p(\mathcal{N};
\ell_{2}^{rc})}\leq C_{p,d} \|f\|_{p}, \; \forall f\in L_{p}(\mathcal N).$$
\end{itemize}
\end{thm}

\begin{remark}\label{app}
\emph{The three estimates in Theorem \ref{t5} for infinite $T_k$'s or summations  over $k\in\mathbb Z$ should be understood as the consequences of the corresponding uniform estimates for all finite subsequences of operators $T_k$'s and some standard approximation arguments (see for instance Section 6.A of \cite{JMX}). For this reason, as in \cite{MP} we will not explain the convergence of infinite sums appearing in the whole paper when there is no ambiguity.}
\end{remark}

The assertion (iii) could be regarded as a result in vector-valued harmonic analysis with the underlying Banach spaces being noncommutative $L_p$ spaces, but seems new even in that setting since the kernel of $T_k$ does not enjoy regularity; while some regularity assumption is required in the theory of vector-valued Calder\'on-Zygmund singular integrals, see e.g. the book \cite{HNVW} and the references therein.

If we set $R_{k}f=(M_{k}-M_{k-1})f$ and $Rf=(R_{k}f)$, then together with the noncommutative Burkholder-Gundy inequality \cite{Ran, Ran05}, Theorem \ref{t5} finds its first application:
\begin{cor}\label{cor1}
For $f\in L_{1}(\mathcal N)$, we have
    $$\|(R_{k}f)\|_{L_{1,\infty}(\mathcal{N};
\ell_{2}^{rc})}\leq C_d\|f\|_{1},$$
where the positive constant $C_d$ depends only on the dimension
$d$.
\end{cor}
Moreover, together with Cuculescu's noncommutative weak type $(1,1)$ maximal estimate for martingales \cite{Cuc}, a similar endpoint estimate for the Hardy-Littlewood  maximal function, firstly established in \cite{M}, follows as a corollary of Theorem \ref{t5}.
\begin{cor}\label{cor2}
For $(f,\lambda) \in L_1(\mathcal{N}) \times \R_+$, there exists a projection
$q \in \mathcal{N}$ with $$\sup_{k\in\Z} \big\|
q M_kf q \big\|_{\infty} \leq
\lambda \qquad \mbox{and} \qquad \lambda\varphi \big( \1_\mathcal{N} -
q \big) \leq C_d\|f\|_1,$$
where the positive constant $C_d$ depends only on the dimension
$d$.
\end{cor}

\begin{remark}
\emph{(i) The $L_p$-versions of the two corollaries ($1<p<\infty$) also hold true if we appeal to the noncommutative Burkholder-Gundy inequalities \cite{PX1} and Doob maximal inequalities \cite{J1}. Moreover, as Theorem \ref{t5}~(iii), the $L_p$-versions of Corollary \ref{cor1} seem new even in the framework of vector-valued harmonic analysis.}

\emph{(ii) Replacing the domain $\mathbb R^d$ by $\mathbb Z^d$ in Theorem \ref{t5}, Corollary \ref{cor1} and \ref{cor2}, similar results hold also true (see e.g. \cite{HM1,JRW03} and the references therein). Then by the noncommutative Calder\'on transference principle \cite{Hon}, we provide another proof of ergodic maximal inequalities associated with actions of groups $\mathbb R^d$ and $\mathbb Z^d$ (see \cite{HLW} for more results).}

\emph{(iii) Note that in \cite{M}, the author established the result in Corollary \ref{cor2} by appealing to $d+1$ noncommutative martingales, while our method involves only one martingale which certainly have further application.}
\end{remark}
Let us briefly analyze the proof of Theorem \ref{t5}. The result for $p=2$ follows trivially from the corresponding commutative result, but we prefer to provide a noncommmutative proof in the Appendix for warming up. For $1\leq p<2$, using the noncommutative Khintchine inequalities in $L_{1,\infty}$ space \cite{C1} and in $L_{p}$ space \cite{LG}, we are reduced to showing the weak type $(1,1)$ and strong type $(p,p)$ estimates of the following operator
\begin{align}\label{finite}
{T}f(x)=\sum_{k}\varepsilon_{k}(M_{k}-\mathsf{E}_k)f(x),
\end{align}
where $(\varepsilon_{k})$ is a Rademacher sequence on a probability space $(\Omega,P)$. Here in the definition of ${T}$ the summation is actually taken over an arbitrarily fixed subsequences of $\mathbb Z$ in terms of Remark \ref{app}. Note that the linearity of the operator ${T}$ allows us to deduce the result for intermediate $p$'s from
the weak type $(1,1)$ and strong type $(2,2)$ estimates by real interpolation. On the other hand, the results for $2<p<\infty$ follow by complex interpolation from $(L_\infty,\mathrm{BMO})$ and strong type $(2,2)$ estimates. But this time the linear operators are $\sum\limits_{k} T_k \otimes e_{1k}$ and $\sum\limits_{k} T_k \otimes e_{k1}$.  Here again the summation is actually taken over an arbitrarily fixed subsequence of $\mathbb Z$. Thus we are reduced to establishing the two endpoint estimates for $p=1,\infty$.

However, with a moment's thought, there are many difficulties to adapt the arguments in \cite{JP1,MP,C} to our setting. Indeed, it is obvious that the kernel associated with ${T}$ (or $T_k$) does not enjoy Lipschitz's regularity while the methods in \cite{JP1,MP,C} depend heavily on this smoothness condition. This prompted us to look for some new methods. It turns out that the main ingredient in showing the strong type $(2,2)$ estimate---an almost orthogonality principle plays an important role in overcoming these difficulties. But numerous modifications are necessary in establishing the noncommutative endpoint estimates.


We end our introduction with a brief description of the organization of the paper. In Section 2, we present some preliminaries on noncommutative $L_{p}$-spaces and introduce some notations. A large portion of Section 3 is devoted to the proof of  conclusion (i) of Theorem \ref{t5} while Corollary \ref{cor1} and Corollary \ref{cor2} will be proved at the end of this section. The $(L_{\infty},\mathrm{BMO})$ estimate is proved in Section 4. In Section 5, we give the proof of conclusion (iii) of Theorem \ref{t5}.

\bigskip


\section{Preliminaries}
This section collects all the necessary preliminaries for the whole paper. The reader is referred to \cite{P2} for more information on noncommutative $L_p$-spaces and noncommutative martingales.
\subsection{Noncommutative $L_p$ spaces}
Let $\M$ be a von Neumann algebra equipped with a \emph{n.s.f.} trace $\tau$. Denote by ${{\mathcal S_\M}}_+$ the set of all $x\in \M_+$ such that $\tau(\mathrm{supp}~x)<\infty$, where $\mathrm{supp} ~x$ denotes the support of $x$ which is the smallest
projection $e$ such that $exe=x$. Let ${\mathcal S_\M}$ be the linear span of ${\mathcal S_\M}_+$. Then ${\mathcal S_\M}$ is
a w*-dense $*$-subalgebra of $\M$. Given $1\leq  p < \infty$, we define
$$\|x\|_p=[\tau(|x|^p)]^{1/p}, \qquad x\in {\mathcal S_\M},$$
where $|x| = (x^*x)^{1/2}$ is the modulus of $x$. Then $({\mathcal S_\M}, \|\cdot\|_p)$ is a normed  space, whose completion is the noncommutative $L_p$-space associated with $(\M, \tau)$, denoted by $L_p(\M, \tau)$ or simply by $L_p(\M)$. For convenience, we set $L_\infty(\M, \tau) =\M$ equipped with the operator norm. Like the classical
$L_p$-spaces, the noncommutative $L_{p}$-spaces behave well with respect to duality and interpolation.


The noncommutative weak $L_1$-space
$L_{1,\infty}(\mathcal{M})$ is defined as the subspace of all $\tau$-measurable operators affiliated with $\mathcal M$ equipped with finite quasi-norm
$$\|x\|_{1,\infty} =\sup_{\lambda > 0}\lambda\tau (|x|> \lambda ):=\sup_{\lambda > 0}\lambda\tau ( \chi_{(\lambda,\infty)}
(|x|) ).$$
It was already shown in \cite[Lemma 2.1]{JRWZ} that for $x_1, x_2 \in
L_{1,\infty}(\M)$ and $\lambda\in\R_{+}$ $$\lambda \, \tau ( |x_1+x_2| > \lambda
)\leq \lambda \, \tau ( |x_1| > \lambda/2 ) +
\lambda \, \tau (|x_2| > \lambda/2 ).$$

\subsection{Noncommutative martingales}
Consider a von Neumann subalgebra $\M_k$ of $\mathcal{M}$ such that $\tau|_{\M_k}$ is semi-finite. Then there exists a unique map $\mathcal{E}_{k}: \mathcal{M} \to
\M_k$ satisfying the following properties:
\begin{itemize}
\item $\mathcal{E}_{k}$ is a normal positive
contractive projection from $\mathcal{M}$ onto $\M_k$;
\item bimodule property,
\[ \mathcal{E}_{k}(x_1 x \hskip1pt x_2) = x_1 \mathcal{E}_{k}(x) \hskip1pt
x_2 \quad \mbox{for all} \quad x_1, x_2 \in \M_k \
\mbox{and} \ x \in \mathcal{M};\]
\item trace preserving: $\tau \circ
\mathcal{E}_{k} = \tau$.
\end{itemize}
The map $\mathcal{E}_{k}$ is called the \emph{conditional expectation} from $\M$  onto $\M_k$. It follows that $\mathcal{E}_k^{\prime}s$ satisfy
$$\forall\ k,j\geq1,\ \mathcal{E}_k\mathcal{E}_j=\mathcal{E}_j\mathcal{E}_k=\mathcal{E}_{\min(k,j)}.$$
Note that for
every $1 \le p < \infty$ and $k \ge 1$, $\mathcal{E}_k$ extends to
a positive contraction $\mathcal{E}_k: L_p(\mathcal{M}) \to
L_p(\mathcal{M}_k)$. We call a \emph{filtration} of $\M$ a sequence of increasing von Neumann subalgebras  $(\M_k)_{k \ge 1}$ such that $\cup_{k}\M_k$ is weak$^\ast$ dense in $\mathcal{M}$ and $\tau|_{\M_{k}}$ is semi-finite for every $k\ge 1$.
Let $1\leq p\leq\infty$. A sequence $x = (x_k)_{k \ge 1}$ in $L_{p}(\M)$ is called a \emph{noncommutative $L_p$-bounded martingale} with
respect to the filtration $(\mathcal{M}_k)_{k \ge 1}$, if
\[
\mathcal{E}_j(x_k) = x_j \quad \mbox{for all} \quad 1 \le j \le k
< \infty,
\]
and $\|x\|_p = \sup\limits_{k \ge 1} \|x_k\|_p < \infty$.
Moreover, $x$ is said to be positive if $x_{k}\geq0$ for all $k\ge 1$. For every $k\ge1$, we define $dx_k = x_k - x_{k-1}$ with the convention that $x_0 = 0$.
The sequence $dx =(dx_k)_{k\ge 1}$ is called the martingale difference sequence of $x$.
\subsection{General notations}
In this subsection, we need to set up some notations that will remain fixed through the paper. Let $\M$ be a
semi-finite von Neumann algebra equipped with a \emph{n.s.f.} trace $\tau$. We consider the tensor von Neumann algebra $\mathcal N=L_{\infty}(\R^{d})\overline{\otimes}\M$ equipped with the tensor \emph{n.s.f.} trace
$\varphi$. Note that for every $1\leq p<\infty$,
$$L_{p}(\mathcal{N};\ \varphi)\cong\ L_{p}(\R^{d};L_{p}(\M)).$$
The space on the right-hand side is the space of Bochner $p$-integrable functions from $\R^{d}$ to $L_{p}(\M)$.
For $1\leq p\leq\infty$, we simply write $L_p(\mathcal{N})$ for the noncommutative
$L_p$ space associated to the pair
$(\mathcal{N},\varphi)$ and
$\|\cdot\|_{p}$ denotes the norm of $L_p(\mathcal{N})$. But if any other $L_{p}$-space appears in a same context, we will precisely mention the associated $L_{p}$-norm in order to avoid possible ambiguity. The lattices of projections are denoted by $\M_\pi$
and $\mathcal{N}_{\pi}$, while $\1_{\M}$ and $\1_{\mathcal{N}}$ stand for the unit
elements.

Denote by $\Q$ the set of all standard dyadic cubes in $\R^d$. The side length of $Q$ is denoted by $\ell(Q)$. Given an integer $k \in \Z$, $\Q_k$ will denote the set of dyadic cubes of side length $2^{-k}$. Let $|Q|=2^{-dk}$ be the volume of such a cube. If $Q\in\Q$ and $f: \R^d \to
\M$ is integrable on $Q$, we define its average over $Q$ as
$$f_Q = \frac{1}{|Q|} \int_Q f(y) \, dy.$$

For $k\in\Z$, let $\sigma_{k}$ be the $k$-th dyadic $\sigma$-algebra, i.e., $\sigma_{k}$ is generated by the dyadic cubes with side length equal to $2^{-k}$. Denote by $\mathsf{E}_k$ the
conditional expectation associated to the classical dyadic
filtration $\sigma_{k}$ on $\R^d$. We also use $\mathsf{E}_k$ for the
tensor product $\mathsf{E}_k \otimes id_\M$ acting on $\mathcal{N}$. If $1
\le p \leq\infty$ and $f \in L_p(\mathcal{N})$, we have
$$\mathsf{E}_k(f) = \sum_{Q \in \Q_k}^{\null} f_Q 1_Q.$$
Similarly, $(\mathcal{N}_k)_{k \in \Z}$ will stand for the corresponding
filtration and $\mathcal{N}_k = \mathsf{E}_k(\mathcal{N})$. For convenience, we will write $f_{k}:=\mathsf{E}_k(f)$ and $\Delta_{k} (f):=f_{k}-f_{k-1}=: df_{k}$.

For all $x\in\R^{d}$, sometimes we write $Q_{x,k}$ for the cube in $\Q_k$ containing $x$, and its center is denoted by $c_{x,k}$. For any odd positive integer $i$ and $Q$ in $\Q_{k}$, let $iQ$ be the cube with the same center as $Q$ such that $\ell(iQ)=i\ell(Q)$.
Notice that for all $x,y \in \R^{d}$ and $k\in\Z$, $x \in iQ_{y,k} \Leftrightarrow y \in iQ_{x,k}$.

Throughout the paper we use the notation
$X\lesssim Y$ for nonnegative quantities $X$ and $Y$ to mean $X\le CY$ for some inessential constant $C>0$.
Similarly, we use the notation $X\backsimeq Y$ if both $X\lesssim Y$ and $Y \lesssim X$ hold.

\section{Weak type $(1,1)$ estimates}
In this section, we first prove conclusion $(i)$ of Theorem \ref{t5} and Corollary \ref{cor1} as well as Corollary \ref{cor2} will be shown in the last subsection. By decomposing $f = f_{1}-f_{2} +i(f_{3}-f_{4})$ with positive $f_{j}$ and $\|f_{j}\|_{1}\leq\|f\|_{1}$ for $j=1,2,3,4$, we assume that $f$ is positive in order to avoid unnecessary computations. Let
us work on the following dense subset of $L_1(\mathcal N)_+$
$$\mathcal N_{c,+} = L_1(\mathcal N) \cap \Big\{
f: \R^d \to \M \, \big| \ f \in \mathcal N_+, \
\overrightarrow{\mathrm{supp}} \hskip1pt f \ \ \mathrm{is \
compact} \Big\} \subset L_1(\mathcal N)_+.$$ Here
$\overrightarrow{\mathrm{supp}}$ means the support of $f$ as an
operator-valued function on $\R^d$. That is to say,
$\overrightarrow{\mathrm{supp}} \hskip1pt f = \mathrm{supp}
\hskip1pt \|f\|_{L_1(\M)}$. We use this terminology to distinguish from $\mathrm{supp} \, f$, which is a projection in $\mathcal N$.
By the standard density argument and the fact that $\mathcal N_{c,+}$ is dense in $L_1(\mathcal N)_{+}$, it suffices to show the desired estimates for $f \in \mathcal{N}_{c,+}$. Moreover, it can be seen that for any $f \in \mathcal{N}_{c,+}$ and $\lambda>0$, there exists $m_{\lambda}(f)\in\Z$ such that $f_{k}\leq\lambda\1_{\mathcal{N}}$ for all $k\leq m_{\lambda}(f)$.

In the remaining part on the proof of Theorem \ref{t5}~(i),  both $f \in \mathcal{N}_{c,+}$ and $\lambda\in(0,+\infty)$ will be fixed and without loss of generality $m_{\lambda}(f)$ is assumed to be $0$.

\subsection{Calder{\'o}n-Zygmund decomposition}

Under the assumption that $m_{\lambda}(f)=0$, applying Cuculescu's construction \cite{Cuc} to the martingale $(f_{k})_{k\geq0}$ relative to the dyadic filtration $(\mathcal{N}_k)_{k \geq0}$ (see for instance \cite[Lemma 3.1]{JP1}), there exists a sequence of projections $(q_k)_{k\in\Z}$ defined by $q_k = \1_\mathcal{N}$ for $k\leq0$ and recursively for $k>0$,
$$q_k=q_k(f,\lambda)=\1_{(0,\lambda]}(q_{k-1} f_k q_{k-1})$$
 such that
\begin{itemize}
\item[(i)] $q_k$ commutes with $q_{k-1} f_k
q_{k-1}$;

\item[(ii)] $q_k$ belongs to $\mathcal{N}_k$ and $q_k f_k q_k \le \lambda \hskip1pt
q_k$;

\item[(iii)] the following estimate holds $$\varphi \Big(
\mathbf{1}_\mathcal{N} - \bigwedge_{k \in\Z} q_k \Big) \le
\frac{ \|f\|_1}{\lambda}.$$
\end{itemize}
Next we define the sequence $(p_k)_{k \in \Z}$ of pairwise disjoint projections by $p_k =
q_{k-1}-q_k$, so that $$\sum_{k \in \Z} p_k = \1_\mathcal{N} - q \quad \mbox{with} \quad q = \bigwedge_{k \in
\Z} q_k.$$
Then we obtain the Calder\'on-Zygmund decomposition of $f$: $f = g_d + g_\mathit{off} + b_d +
b_\mathit{off}$ with
$$\begin{array}{rclcrcl} \displaystyle g_d & = & \displaystyle
qfq + \sum_{k \in \Z} p_k f_k p_k, & \quad & g_\mathit{off} & = &
\displaystyle \sum_{i \neq j} p_i f_{i \vee j} p_j \ + \ q f
q^\perp + q^\perp f q, \\ [15pt] b_d & = & \displaystyle \sum_{k
\in \Z} p_k \hskip1pt (f - f_k) \hskip1pt p_k, & \quad &
b_{\mathit{off}} & = & \displaystyle \sum_{i \neq j} p_i (f-f_{i
\vee j}) p_j
\end{array}$$ where $i \vee j = \max (i,j)$ and $q^\perp =
\1_\mathcal{N} - q$, satisfying the properties that we collect in the following (see \cite{JP1,C} for more details).

\begin{lem}[\cite{JP1}]\label{r1}
The following diagonal estimates hold
\begin{equation}
\| g_d \|_2^2 \le 2^d
\lambda \, \|f\|_1 \quad \mbox{and} \quad \|
b_d \|_1 \le 2 \, \|f\|_1.
\end{equation}
\end{lem}
\begin{lem}[\cite{C}] \label{r2}
One can reorganize $g_{\mathit{off}}$ as  $$g_{\mathit{off}} = \sum_{s=1}^\infty
\sum_{k=1}^\infty p_k df_{k+s} q_{k+s-1} + q_{k+s-1}
df_{k+s} p_k \triangleq \sum_{s=1}^\infty \sum_{k=1}^\infty
(g^{\ell}_{s,k}+g^{r}_{s,k}) \triangleq \sum_{s=1}^\infty (g^{\ell}_{s}+g^{r}_{s})$$ with
$$\sup_{s \ge 1} \|g^{\ell}_{s}\|_2^2 = \sup_{s \ge 1}
\sum_{k=1}^\infty \|g^{\ell}_{s,k}\|_2^2 \lesssim \lambda \,
\|f\|_1$$ and $\Delta_{k+s}(g^{\ell}_{s}) =g^{\ell}_{s,k}$ and the same conclusions for $g^{r}_{s}$.
\end{lem}

\begin{lem}[\cite{C}]\label{L3}
If one sets $p_{Q}:=p_{k}(x)$ for all $k$, $Q\in\Q_{k}$ and any $x\in Q$, and define
\begin{center}
$\zeta = \big(\bigvee\limits_{Q\in \Q} p_Q\1_{5Q}\big)^{\bot}$,
\end{center}
then
\begin{enumerate}
\item $\varphi(1-\zeta) \leq 5^d\dfrac{\|f\|_1}{\lambda}$;
\item for all cubes $Q \in \Q$, we have the following cancellation property:
\begin{center}
$x \in 5Q \Rightarrow \zeta(x)p_Q=p_Q\zeta(x) = 0$.
\end{center}
\end{enumerate}
\end{lem}
\begin{lem}[\cite{C}]\label{L5}
If one sets $b_{i,j} = p_i(f-f_{i\vee j})p_j$ for all $i,j \in \Z$, then the following cancellation properties hold:
\begin{enumerate}
\item for all $i,j\in\Z$ and $Q\in \Q_{i\vee j}$, $\int_Q b_{i,j} = 0$;
\item for all $x,y\in\R^{d}$ such that $y \in 5Q_{x,i\wedge j}$, $\zeta(x)b_{i,j}(y)\zeta(x) = 0$.
\end{enumerate}
\end{lem}

\subsection{Some preliminary reductions and technical lemmas}

To prove the weak type $(1,1)$ boundedness of $(T_k)$---a square function estimate, we need a noncommutative Khintchine inequality in $L_{1,\infty}$ for a Rademacher sequence $(\varepsilon_{k})$ on a fixed probability space $(\Omega,P)$, which was essentially established by Cadilhac \cite[Corollary 3.2]{C1}, to linearize the underlying operator.
\begin{lem}[\cite{C1}]\label{C1}
For any finite sequence $(u_{k})$ in $L_{1,\infty}(\mathcal{N})$, we have
$$\big\|\sum_{k}\varepsilon_{k}u_{k}\big\|_{L_{1,\infty}(L_{\infty}(\Omega)\overline{\otimes}
\mathcal N)}\backsimeq\|(u_{k})\|_{L_{1,\infty}(\mathcal{N};
\ell_{2}^{rc})}.$$
\end{lem}
Recalling
\begin{equation}\label{9898}
{T}f(x)=\sum_{k}\varepsilon_{k}T_kf(x)=\sum_{k}\varepsilon_{k}(M_{k}-\mathsf{E}_k)f(x),
\end{equation}
where the summation is taken over a fixed finite subset of $\mathbb Z$ (see the remark after \eqref{finite}),
we immediately obtain the following corollary from Lemma \ref{C1}.
\begin{cor}\label{inte}
Let $h\in \mathcal{N}_{c,+}$. Then we have
$$\|(T_{k}h)\|_{L_{1,\infty}(\mathcal{N};\ell_{2}^{rc})}\backsimeq
\|{T}h\|_{L_{1,\infty}(L_{\infty}(\Omega)\overline{\otimes}\mathcal N)}.$$
\end{cor}

With Corollary \ref{inte}, it suffices to verify $$\|Tf\|_{L_{1,\infty}(L_{\infty}(\Omega)\overline{\otimes}\mathcal N)}\lesssim\|f\|_{1}.$$ Applying the distributional inequality (see e.g. \cite[Lemma 2.1]{JRWZ}) we have
$$\widetilde{\varphi}\Big(|Tf|>\lambda
\Big)\leq\widetilde{\varphi}\Big(|Tb_{d}|>\frac{\lambda}{4}
\Big)+\widetilde{\varphi}\Big(|Tb_{\mathit{off}}|>\frac{\lambda}{4}
\Big)+
\widetilde{\varphi}\Big(|Tg_{d}|>\frac{\lambda}{4}
\Big)+\widetilde{\varphi}\Big(|Tg_{\mathit{off}}|>\frac{\lambda}{4}
\Big),$$
where $\widetilde{\varphi}=\int_{\Omega}\otimes\varphi$. Hence, it suffices to prove
$$\widetilde{\varphi}\Big(|Th|>\lambda
\Big)\lesssim\frac{\|f\|_{1}}{\lambda}$$
for $h=b_{d}$, $b_{\mathit{off}}$, $g_{d}$ and $g_{\mathit{off}}$.

\bigskip

The following two lemmas will be frequently used in the rest of the proof.

The first one is an almost orthogonality principle, which is well-known in classical harmonic analysis, see for instance \cite{JRW03, HM1}.
\begin{lem}\label{L1}
Let $S_{k}$ be a bounded linear map on $L_2$ for each $k\in \Z$ and $h\in L_2$.
If $(u_{n})_{n\in\Z}$ and $(v_{n})_{n\in\Z}$ are two sequences of functions in $L_{2}$ such that $h=\sum\limits_{n\in\Z}u_{n}$ and $\sum\limits_{n\in\Z}\|v_{n}\|_{2}^{2}<\infty$, then
$$\sum_{k}\|S_{k}h\|_{2}^{2}\leq w^{2}\sum_{n\in\Z}\|v_{n}\|_{2}^{2}$$
provided that there is a sequence $(\sigma(j))_{j\in \Z}$ of positive numbers with $w=\sum\limits_{j\in\Z}\sigma(j)<\infty$ such that
$$\|S_{k}(u_{n})\|_{2}\leq \sigma(n-k)\|v_{n}\|_{2}$$
for every $n,k$.
\end{lem}

Let $n\in\mathbb Z$ and $B\subset\R^{d}$ be a Euclidean ball, we define
\begin{eqnarray}\label{C18}
\mathcal{I}(B,n)=\bigcup_{{\begin{subarray}{c}
Q\in\Q_{n} \\ \partial B\cap Q\neq \emptyset
\end{subarray}}} Q\cap B.
\end{eqnarray}
Then for integer $k<n$, the linear operator $M_{k,n}: L_1(\mathcal N)\rightarrow L_1(\mathcal N)$ is defined as
\begin{align}\label{mkn}
M_{k,n}h(x)=\frac{1}{|B_k|}\int_{\mathcal{I}(B_k+x,n)}h(y)dy.
\end{align}

\begin{lem}\label{kn}
Let $M_{k,n}$ be the linear operator defined as above for $k<n$. Then for $1\leq p\leq\infty$ and $h\in L_p(\mathcal N)$, we have
$$\|M_{k,n}h\|_{p}\lesssim 2^{k-n}\|h\|_p;$$
if moreover $h$ is positive, then
$$\|M_{k,n}h\|_{p}\lesssim 2^{k-n}\|h_n\|_p$$
where $h_n=\mathsf{E}_n(h)$.
\end{lem}

\begin{proof}
Let $p\in[1,\infty]$ and $p^\prime$ be its conjugate index and $h\in L_p(\mathcal N)$. By the Minkowski and H\"older inequalities
\begin{align*}\|M_{k,n}h(x)\|_{L_p(\mathcal M)}&\leq \frac{1}{|B_k|}\int_{\mathcal{I}(B_k+x,n)}\|h(y)\|_{L_p(\mathcal M)}dy\\
&\leq \frac{|\mathcal{I}(B_k+x,n)|^{\frac{1}{p^\prime}}}{|B_k|}\big(\int_{\mathcal{I}(B_k+x,n)}\|h(y)\|^p_{L_p(\mathcal M)}dy\big)^{\frac1p}.
\end{align*}
Note that the measure of $\mathcal{I}(B_k+x,n)$ is not more than a constant multiple of $2^{-n}2^{(d-1)(-k)}$.
Taking power $p$ and integrating over $\mathbb R^d$, one gets by Fubini's theorem
\begin{align*}
\|M_{k,n}h\|^p_{p}\lesssim  \frac{(2^{-n}2^{(d-1)(-k)})^{\frac{p}{p^\prime}}}{(2^{-kd})^p}2^{-n}2^{(d-1)(-k)}\|h\|_p^p=2^{(k-n)p}\|h\|_p^p,
\end{align*}
which gives the first estimate.

If moreover $h\in L_p(\mathcal N)$ is positive, one has
$$M_{k,n}h(x)\leq \frac{1}{|B_k|}\bigcup_{{\begin{subarray}{c}
Q\in\Q_{n} \\ \partial B_k+x\cap Q\neq \emptyset
\end{subarray}}} \int_{Q}h(y)dy= \frac{1}{|B_k|}\bigcup_{{\begin{subarray}{c}
Q\in\Q_{n} \\ \partial B_k+x\cap Q\neq \emptyset
\end{subarray}}} \int_{Q}h_n(y)dy.$$
By the fact that the measure of the union of the dyadic cubes in $\Q_{n}$ which intersects with the boundary of $x+B_{k}$ is not more than a constant multiple of $2^{-n}2^{(d-1)(-k)}$, the same argument as above yields the second estimate.
\end{proof}

\subsection{Estimates for the bad function}

\subsubsection{Estimate for $Tb_d$}
Let us now
prove the assertion for $b_d$.
Using the
projection $\zeta$ introduced in Lemma \ref{L3}, we
consider the following decomposition $$Tb_d = (\1_\mathcal{N} - \zeta) T
b_d (\1_\mathcal{N} - \zeta) + \zeta \hskip1pt T b_d (\1_\mathcal{N} - \zeta)
+ (\1_\mathcal{N} - \zeta) T b_d \zeta + \zeta \hskip1pt T b_d
\zeta.$$
Therefore, Lemma \ref{L3} gives:
\begin{align*}
\hskip1pt \widetilde{\varphi} \Big(|Tb_d|>\frac{\lambda}{4}
\Big)
\lesssim&\ \hskip1pt \varphi (\1_\mathcal{N} - \zeta) +
\hskip1pt \widetilde{\varphi} \Big(|\zeta Tb_d\zeta|>\frac{\lambda}{16}
\Big)\\
\lesssim&\ \frac{ \|f\|_{1}}{\lambda}+\widetilde{\varphi} \Big(|\zeta Tb_d\zeta|>\frac{\lambda}{16}
\Big).
\end{align*}
Hence, our aim is to estimate $\widetilde{\varphi} \Big(|\zeta Tb_d\zeta|>\frac{\lambda}{16}
\Big)$.

\begin{prop}\label{C8}
The following estimate holds true
$$\lambda\widetilde{\varphi} \Big(|\zeta Tb_{d}\zeta|>\frac{\lambda}{16}
\Big)\lesssim\|f\|_{1}.$$
\end{prop}
\begin{proof}
By Chebychev's inequality,
$$\widetilde{\varphi} \Big(|\zeta Tb_{d}\zeta|>\frac{\lambda}{16}
\Big)\lesssim\frac{\|\zeta Tb_{d}\zeta\|^{2}_{L_{2}(L_{\infty}(\Omega)\overline{\otimes}\mathcal N)}}{\lambda^{2}}.$$
Therefore, it suffices to show
\begin{eqnarray}\label{9}
\|\zeta Tb_{d}\zeta\|^{2}_{L_{2}(L_{\infty}(\Omega)\overline{\otimes}\mathcal N)}\lesssim\lambda^{2}\sum_{n\in\Z}\|p_{n}\|^{2}_{2},
\end{eqnarray}
since $p_{n}=0$ for $n\leq0$ and  $$\sum_{n\in\Z}\|p_{n}\|^{2}_{2}=\sum_{n=1}^{\infty}\|p_{n}\|_{1}\lesssim\frac{\|f\|_{1}}{\lambda}.$$ To estimate (\ref{9}), we first note that the orthogonality of $\varepsilon_{k}$ implies
\begin{eqnarray*}\label{333}
\|\zeta Tb_{d}\zeta\|^{2}_{L_{2}(L_{\infty}(\Omega)\overline{\otimes}\mathcal N)}=
\sum_{k}
\|\zeta \, T_k \hskip-1pt b_d \, \zeta
\|^{2}_{2}=\sum_{k}\|
\zeta \, (M_k-\mathsf{E}_k) \hskip-1pt b_d \, \zeta\|^{2}_{2}.
\end{eqnarray*}
Thus (\ref{9}) is equivalent to
\begin{eqnarray}\label{bad}
\sum_{k}\|
\zeta(M_k-\mathsf{E}_k)b_{d}\zeta\|^{2}_{2}\lesssim\lambda^{2}\sum_{n\in\Z}\|p_{n}\|^{2}_{2}.
\end{eqnarray}
If we set $b_{n}$ to be $p_n
(f-f_n) p_n$, then $b_{d}=\sum\limits_{n=1}^{\infty}b_{n}$. To prove (\ref{bad}), taking $S_{k}h=\zeta(M_k-\mathsf{E}_k)h\zeta$, $u_{n}=b_{n}$ and $v_{n}=p_{n}$ in Lemma \ref{L1}, we are reduced to showing
\begin{eqnarray}\label{bad1}
\|\zeta(M_k-\mathsf{E}_k)b_{n}\zeta\|^{2}_{2}\lesssim2^{-2|k-n|}\lambda^{2}\|p_{n}\|^{2}_{2}.
\end{eqnarray}

First, we claim that for $k\geq n$,
\begin{align}\label{key ob}
\zeta(x) (M_k-\mathsf{E}_k) b_{n}(x)\zeta(x)=0,\;\forall x\in\R^{d}.
\end{align}
Indeed,  by the cancellation property---Lemma \ref{L5}~(2), we have
$$\zeta(x)M_{k}b_{n}(x)\zeta(x)=\zeta(x)\frac1{|B_{k}|}\int_{x+B_{k}}b_n(y)\1_{y\notin 5Q_{x,n}}dy=0,$$
since $x+B_{k}\subset 5Q_{x,n}$; similarly
$$
\zeta(x)\mathsf{E}_k b_{n}(x)\zeta(x)=\zeta(x)\frac{1}{|Q_{x,k}|} \int_{Q_{x,k}}b_{n}(y)\1_{y\notin 5Q_{x,n}}dy\zeta(x)=0,
$$
since $Q_{x,k}\subset Q_{x,n}$.

Now we turn to the proof of \eqref{bad1} in the case of $k<n$. Observe that $\mathsf{E}_k b_{n}(x)=\mathsf{E}_k\mathsf{E}_nb_{n}(x)=0$ follows from the cancellation property of $b_n$---Lemma \ref{L5}~(1). Thus it suffices to show in the present case
\begin{eqnarray}\label{0001}
\|M_k b_{n}\|_{2}\lesssim2^{k-n}\lambda\|p_{n}\|_{2}.
\end{eqnarray}
By the cancellation property of $b_n$, it is easy to check that
$$M_kb_n=M_{k,n}b_n$$
where $M_{k,n}$ was defined in \eqref{mkn}. Furthermore, write $b_n=p_nfp_n-p_n
f_n p_n$ as the difference of two positive operators. Then by the triangle inequality, Lemma \ref{kn} and the fact that $\mathsf{E}_n(p_{n}fp_{n})=p_{n}f_{n}p_{n}$, we get
\begin{align*}
\|M_k b_{n}\|_{2}&=\|M_{k,n} b_{n}\|_{2}\lesssim 2^{k-n}\|p_nf_np_n\|_2.
\end{align*}
Noting that from Cuculescu's construction, one has $p_{n}f_{n}p_{n}\lesssim\lambda p_n$ (see also \cite[Page 561]{JP1}), which
 gives
 $$\|p_nf_np_n\|_2\lesssim \lambda \|p_n\|_2.$$
 This yields the desired estimate \eqref{0001}.
\end{proof}

\begin{rk}
\emph{We should point out that the above method based on a 2-norm estimate seems no longer applicable to the off-diagonal term of $b$, as we have seen that the diagonal estimate and the positivity play important roles in the above argument, while the off-diagonal part of $b$ does not enjoy these properties.}
\end{rk}

\subsubsection{Estimate for $Tb_\mathit{off}$}
Let us now consider the off-diagonal term
$b_\mathit{off}$ . As for $Tb_{d}$, it suffices to estimate $\zeta Tb_{\mathit{off}}\zeta$.

\begin{prop}\label{C11}
The following estimate holds true,
$$\lambda\widetilde{\varphi} \Big(|\zeta Tb_{\mathit{off}}\zeta|>\frac{\lambda}{16}
\Big)\lesssim\|f\|_{1}.$$
\end{prop}

\begin{proof}
Rewrite $b_{\mathit{off}}$ as
$$b_\mathit{off}=\sum\limits_{s=1}^{\infty}\sum\limits_{n=1}^{\infty}b_{n,s},$$ where
$$b_{n,s}
= p_n (f - f_{n+s}) p_{n+s} + p_{n+s} (f - f_{n+s}) p_n.$$
By the Chebychev and Minkowski inequalities, one gets
\begin{align*}
\lambda\widetilde{\varphi} \Big(|\zeta Tb_{\mathit{off}}\zeta|>\frac{\lambda}{16}
\Big)
& \lesssim\|\zeta T
b_{\mathit{off}} \zeta\|_{L_{1}(L_{\infty}(\Omega)\overline{\otimes}\mathcal N)}\\
& \leq\sum_{s=1}^{\infty}
\sum_{n=1}^{\infty}\|\zeta  \sum_{k}\varepsilon_{k}T_k
b_{n,s} \zeta\|_{L_{1}(L_{\infty}(\Omega)\overline{\otimes}\mathcal N)}\\
& \leq \sum_{s=1}^\infty\sum_{n=1}^\infty\sum_{k}\Big\|\zeta (M_k-\mathsf{E}_k)  b_{n,s} \, \zeta
\Big\|_{1}.
\end{align*}

On the one hand, the same argument as for claim \eqref{key ob} yields $\zeta (M_k-\mathsf{E}_k)  b_{n,s} \, \zeta=0$ for any $k\geq n$; on the other hand, for $k<n$, $\mathsf{E}_k b_{n,s}=0$ follows from the cancellation property of $b_{n,s}$ in Lemma \ref{L5}~(1). Thus to complete the proof of Proposition \ref{C11}, it suffices to show
\begin{align}\label{deboff}
\sum_{s=1}^\infty\sum_{n=1}^\infty\sum_{k:k< n}\| M_k b_{n,s} \|_1\lesssim \|f\|_1.
\end{align}
Applying again the cancellation property of $b_{n,s}$, we have $M_k b_{n,s}=M_{k, n+s}b_{n,s}$. Thus by Lemma \ref{kn},
$$\| M_k b_{n,s} \|_1\lesssim 2^{-s}2^{k-n}\|b_{n,s}\|_1;$$ note that the following estimate has been obtained in Lemma 3.7 of \cite{C}, for all $s\geq1$, we have
$$\sum\limits_{n=1}^{\infty}\|b_{n,s}\|_{1}\lesssim\|f\|_{1}.$$
Therefore, by Fubini's theorem,
\begin{align*}
\sum_{s=1}^\infty\sum_{n=1}^\infty\sum_{k:k< n}\big\| M_k \hskip-1pt b_{n,s} \big\|_1
& \lesssim\sum_{s=1}^\infty\sum_{n=1}^\infty\sum_{k:k< n}2^{-s}\cdot2^{k-n}\|b_{n,s}\|_{1}\\
& \lesssim\sum_{s=1}^\infty\sum_{n=1}^\infty2^{-s}\|b_{n,s}\|_{1}\\
& \lesssim\sum_{s=1}^\infty2^{-s}\|f\|_{1}\lesssim\|f\|_{1}.
\end{align*}
This gives estimate \eqref{deboff} and thus completes the argument for  $Tb_{\mathit{off}}$.
\end{proof}

\begin{remark}
\emph{With hindsight, the arguments in dealing with the off-diagonal bad part can be used to handle the diagonal bad part. But we prefer to give the previous proof based on 2-norm estimates for the latter part since we believe that this approach should be more powerful and have further applications as in the case of classical harmonic analysis.}
\end{remark}

\subsection{Estimates for the good function}
\subsubsection{Estimate for $Tg_d$}
To handle $g_d$, we need the fact that $T$ is bounded from $L_{2}(\mathcal N)$ to $L_2(L_{\infty}(\Omega)\overline{\otimes}\mathcal{N})$, which will be proved in the Appendix.
\begin{lem}\label{t1}
Let $h\in L_{2}(\mathcal N)$. Then there exists a constant $C_{d}$ depending only on the dimension
$d$ such that
$$\|Th\|_{L_2(L_{\infty}(\Omega)\overline{\otimes}\mathcal{N})}\leq C_{d}\|h\|_{2}.$$
\end{lem}
\begin{prop}\label{C12}
The following estimate holds true,
$$\lambda\widetilde{\varphi}\ \Big(|Tg_{d}|>\frac{\lambda}{16}
\Big)\lesssim\|f\|_{1}.$$
\end{prop}
\begin{proof}
Together with the diagonal estimate given in Lemma \ref{r1}, Chebychev's inequality and H\"{o}lder's inequality yield
$$\widetilde{\varphi}\ \Big(|Tg_{d}|>\frac{\lambda}{16}
\Big)\lesssim\frac{\|Tg_{d}\|_{L_{2}(L_{\infty}(\Omega)\overline{\otimes}\mathcal N)}}{\lambda^{2}}
\lesssim\frac{\|g_{d}\|_{2}^{2}}{\lambda^{2}}\leq\frac{\|g_{d}\|_{1}\|g_{d}\|_{\infty}}{\lambda^{2}}
\lesssim\frac{\|f\|_{1}}{\lambda}.$$
This completes the proof of our assertion for $Tg_{d}$.
\end{proof}
\subsubsection{A pseudo-localization result}
As mentioned before, to handle the off-diagonal term of $g$, we need to establish a \emph{pseudo-localization principle} in our case. We state this principle as the following theorem.
\begin{thm} \label{t2}
Let $h \in L_{2}(\mathcal N)$ and $s\geq1$ be an integer. For all $k\in \Z$, let $A_k$ and $B_k$ be projections in $\mathcal N_k$ such that $A_k^{\bot}dh_{k+s}=dh_{k+s}B_k^{\bot}=0$. Write:
\begin{center}
$A_k = \bigvee\limits_{{Q\in \Q_k}} A_Q\1_Q$, $A_Q \in \M_\pi$ and define $5A_k= \bigvee\limits_{Q\in \Q_k} A_Q\1_{5Q}$,
\end{center}
Define $5B_k$ in the same way. Let
\begin{align}\label{2}
A_{h,s}= \bigvee_{k \in \Z} 5A_k\ \mbox{and}\ B_{h,s}= \bigvee_{k \in \Z} 5B_k.
\end{align}
Then we have
\begin{center}
$\|A_{h,s}^{\bot}Th\|_{L_{2}(L_{\infty}(\Omega)\overline{\otimes}\mathcal N)} \lesssim 2^{-\frac{s}{2}}\|h\|_{2}\ \mbox{and}\ \|(Th)B_{h,s}^{\bot}\|_{L_{2}(L_{\infty}(\Omega)\overline{\otimes}\mathcal N)} \lesssim 2^{-\frac{s}{2}}\|h\|_{2}.$
\end{center}
\end{thm}

\begin{proof}
We will only deal with the case involving $A_{h,s}$ since the exact same proof works for another case. Let $h=\sum\limits_{n\in\Z}dh_{n}$, where $dh_{n}$ is the martingale difference of $h$. Then
$$\sum_{n\in\Z}\|dh_{n}\|_{2}^{2}=\|h\|_{2}^{2}.$$
Applying the orthogonality of $\varepsilon_{k}$, we have
$$\|A^{\perp}_{h,s}Th\|^{2}_{L_{2}(L_{\infty}(\Omega)\overline{\otimes}\mathcal N)}=\sum_{k}\|A^{\perp}_{h,s}
(M_{k}-\mathsf{E}_{k})\sum_{n}dh_{n}\|_{2}^{2}.$$
Moreover, Lemma \ref{L1} implies that in order to finish the proof of Theorem \ref{t2}, it is enough to show
\begin{align}\label{88}
\|A^{\perp}_{h,s}(M_{k}-\mathsf{E}_k)dh_{n}\|^{2}_{2}\lesssim 2^{-{s}}2^{-{|n-k|}}\|dh_{n} \|^{2}_{2}.
\end{align}

First we claim that for all $k\geq n-s$,
\begin{align}\label{15}
A^{\perp}_{h,s}(x)(M_{k}-\mathsf{E}_k)dh_{n}(x)=0,\;\forall x\in\mathbb R^d.
\end{align}
The similar arguments as for claim \eqref{key ob} works here. Indeed,
from the construction of $A_{h,s}$ in (\ref{2}), it is easy to see that for all $j\in\mathbb Z$, $u,v\in\R^{d}$ with $v\in5Q_{u,j}$, we have $A_{j}(v)\leq (5A_{j})(u)\leq A_{h,s}(u)$
and thus
\begin{align}\label{8}
A^{\perp}_{h,s}(u)=A^{\perp}_{h,s}(u)(5A_{j})^{\perp}(u)=A^{\perp}_{h,s}(u)(5A_{j})^{\perp}(u)A_{j}^{\perp}(v).
\end{align}
Now let $x\in\mathbb R^d$ and $k,n$ be fixed such that $k\geq n-s$. Under this restriction, one has $x+B_{k}\subset 5Q_{x,n-s}$. Thus by \eqref{8} and the assumption $A_{n-s}^{\perp}dh_{n}=0$
\begin{align*}
  A^{\perp}_{h,s}(x)M_{k}dh_{n}(x) & =A^{\perp}_{h,s}(x)\frac{1}{|B_{k}|}
\int_{x+B_{k}}dh_{n}(y)dy\\
& =A_{h,s}^{\perp}(x)\frac{1}{|B_{k}|}
\int_{x+B_{k}}(5A_{n-s})^{\perp}(x)A^{\perp}_{n-s}(y)dh_{n}(y)dy=0.
\end{align*}
Similarly, $A^{\perp}_{h,s}(x)\mathsf{E}_{k}dh_{n}(x)=0$ and thus claim \eqref{15} is verified.

It remains to show \eqref{88} for $k<n-s$. First of all, one has trivially $\mathsf{E}_kdh_{n}=\mathsf{E}_k\mathsf{E}_ndh_{n}=0$. Using the cancellation property of $dh_n$ over all atoms in $\Q_{n-1}$, one gets $M_kdh_n=M_{k,n-1}dh_n$. Noting that $k<n-s\leq n-1$, thus by Lemma \ref{kn},
\begin{align*}
\|A^{\perp}_{h,s}M_{k}dh_{n}\|_{2}\lesssim 2^{k-n}\|dh_n\|_2\lesssim 2^{-\frac s2}2^{\frac{k-n}2}\|dh_n\|_2.
\end{align*}
Therefore, estimate \eqref{88} is completely verified and thus the proof of Theorem \ref{t2} is complete.
\end{proof}

\subsubsection{Estimate for $Tg_\mathit{off}$}
We now consider the off-diagonal term $g_\mathit{off}$. As for $Tb_d$, we decompose
the term $T g_{\mathit{off}}$ into four parts
$$(\1_\mathcal{N} - \zeta) T g_\mathit{off} (\1_\mathcal{N} - \zeta) + \zeta
\hskip1pt T g_\mathit{off} (\1_\mathcal{N} - \zeta) + (\1_\mathcal{N} - \zeta) T
g_\mathit{off} \zeta + \zeta \hskip1pt T g_\mathit{off}
\zeta,$$ where $\zeta$ is the projection constructed in Lemma
\ref{L3} and it suffices to show the corresponding estimate for the last part.

\begin{prop}\label{C10}
The following estimate holds true,
$$\lambda\widetilde{\varphi}\ \Big(|\zeta Tg_\mathit{off} \zeta|>\frac{\lambda}{16}
\Big)\lesssim\|f\|_{1}.$$
\end{prop}
\begin{proof}
 Set $$g^{\ell}\triangleq\sum_{s=1}^\infty
\sum_{k=1}^\infty p_k df_{k+s} q_{k+s-1}= \sum_{s=1}^\infty \sum_{k=1}^\infty
g^{\ell}_{s,k} = \sum_{s=1}^\infty g^{\ell}_{s}$$
and
$$g^{r}\triangleq\sum_{s=1}^\infty
\sum_{k=1}^\infty  q_{k+s-1}df_{k+s} p_k= \sum_{s=1}^\infty \sum_{k=1}^\infty
g^{r}_{s,k} = \sum_{s=1}^\infty g^{r}_{s}$$
then $g_\mathit{off}=g^{\ell}+g^{r}$. Moreover, by the distribution inequality, it suffices to show
$$\lambda\widetilde{\varphi}\ \Big(|\zeta Tg^{\ell}\zeta|>\frac{\lambda}{32}
\Big)\lesssim\|f\|_{1}$$
and similar estimate for the term involving $g^{r}$.

In the following, we just deal with the term involving $g^{\ell}$ since another term can be handled similarly. By the Chebychev inequality,
\begin{align*}
\lambda\widetilde{\varphi} \Big(|\zeta Tg^{\ell}\zeta|>\frac\lambda{32}
\Big) & \lesssim  \frac{1}{\lambda}
\big\| \zeta \, T g^{\ell} \big\|_{L_{2}(L_{\infty}(\Omega)\overline{\otimes}\mathcal N)}^2 \\
& \le  \frac{1}{\lambda} \, \Big( \sum_{s=1}^\infty\big\| \zeta \, T g^{\ell}_{s} \,\big\|_{L_{2}(L_{\infty}(\Omega)\overline{\otimes}\mathcal N)} \Big)^2.
 \end{align*}
Now we are at a position to apply Theorem $\ref{t2}$. Setting $h=g^{\ell}_s$ and $A_k = p_k$ and verifying easily $dh_{k+s}=g^{\ell}_{s,k}$ and $A_{h,s}=\zeta$, we get
$$\|\zeta Tg^{\ell}_{s}\|_{L_{2}(L_{\infty}(\Omega)\overline{\otimes}\mathcal N)} \lesssim 2^{-\frac{s}{2}} \|g^{\ell}_{s}\|_2.$$
Thus
\begin{align*}
\lambda\widetilde{\varphi} \Big(|\zeta Tg^{\ell}\zeta|>\frac\lambda{32}
\Big) & \lesssim
\frac{1}{\lambda} \, \Big( \sum_{s=1}^\infty
\, 2^{-\frac{s}{2}} \,
\|g^{\ell}_{s}\|_2 \Big)^2\\& \lesssim  \frac{1}{\lambda} \, \Big( \sum_{s=1}^\infty
\, 2^{-\frac{s}{2}} \,
\sqrt{\lambda\|f\|_{1}} \Big)^2\lesssim  \|f\|_1,
\end{align*}
where the middle inequality follows from the off-diagonal estimate given in
Lemma \ref{r2}. This is the required estimate involving $g^{\ell}$ and thus finishes the proof of Proposition \ref{C10}.
\end{proof}

\subsection{Conclusion}
Combining Proposition \ref{C8}, \ref{C11}, \ref{C12}, \ref{C10} and the estimate of $\1_\mathcal N-\zeta$ before Proposition \ref{C8}, we obtain the desired weak type $(1,1)$ estimate announced in Theorem \ref{t5}.

\bigskip

At the end of this section, we are at a position to prove the two corollaries announced in the Introduction.
\begin{proof}[Proof of Corollary \ref{cor1}.]
We can write
\begin{align*}
R_{k}f(x)& =(M_{k}-M_{k-1})f(x)\\
& =(M_{k}-\mathsf{E}_{k})f(x)+(\mathsf{E}_{k}-\mathsf{E}_{k-1})f(x)+(\mathsf{E}_{k-1}-M_{k-1})f(x).
\end{align*}
The first and third terms of the above expression can be handled by conclusion (i) of Theorem \ref{t5}, and the middle term has been essentially obtained in \cite[Theorem 3.1]{Ran}. Thus, we get the desired estimate of Corollary \ref{cor1}.
\end{proof}

\bigskip

\begin{proof}[Proof of Corollary \ref{cor2}.]
Let $f\in L_1(\mathcal N)$ and $\lambda>0$.
We first decompose $M_kf$ as
$$M_kf=(M_{k}-\mathsf{E}_{k})f+\mathsf{E}_{k}f.$$
Then Cuculescu's inequality \cite{Cuc} implies that we can find a projection $e_{1}\in\mathcal{N}$ such that
$$
\sup_{k\in\Z} \big\|
e_{1} \mathsf{E}_{k}f e_{1}\big\|_{\infty} \leq
\lambda \qquad \mbox{and} \qquad \lambda\varphi \big( \1_\mathcal{N} -
e_{1} \big) \lesssim\|f\|_1.
$$

On the other hand, conclusion (i) of Theorem \ref{t5} yields that there exists a decomposition  $T_{k}f=g_{k}+h_{k}$ satisfying $$\|(g_{k})\|_{L_{1,\infty}(\mathcal{N}; \ell_{2}^{c})}+\|(h_{k})\|_{L_{1,\infty}(\mathcal{N}; \ell_{2}^{r})}\lesssim\|f\|_{1}.$$
Set $e_{2}=\chi_{(0,\lambda]}\big((\sum\limits_{k}|g_{k}|^{2})^{\frac{1}{2}}\big)$ and $e_{3}=\chi_{(0,\lambda]}\big((\sum\limits_{k}|h^{\ast}_{k}|^{2})^{\frac{1}{2}}\big)$, then
$$
\big\|\big((\sum_{k}|g_{k}|^{2})^{\frac{1}{2}}\big)
e_{2}\big\|_{\infty} \leq
\lambda \qquad \mbox{and} \qquad \lambda\varphi \big( \1_\mathcal{N} -
e_{2} \big) \lesssim\|f\|_1;
$$
and
$$
\big\|\big((\sum_{k}|h^{\ast}_{k}|^{2})^{\frac{1}{2}}\big)
e_{3}\big\|_{\infty} \leq
\lambda \qquad \mbox{and} \qquad \lambda\varphi \big( \1_\mathcal{N} -
e_{3} \big) \lesssim\|f\|_1.
$$
Let $e_{4}=e_{2}\bigwedge e_{3}$. Then
$$\lambda\varphi \big( \1_\mathcal{N} -
e_{4} \big) \lesssim\|f\|_1,$$
and
\begin{align*}
   \big\|e_{4}(M_{k}-\mathsf{E}_{k})fe_{4}\big\|_{\infty}
    \leq&\ \big\|e_{4}g_{k}e_{4}\big\|_{\infty}+\big\|e_{4}h_{k}e_{4}\big\|_{\infty}\\
    =&\ \big\|e_{4}g_{k}e_{4}\big\|_{\infty}+\big\|e_{4}h^{\ast}_{k}e_{4}\big\|_{\infty}\\
    =&\ \big\|e_{4}u_{k}|g_{k}|e_{4}\big\|_{\infty}+\big\|e_{4}v_{k}|h^{\ast}_{k}|e_{4}\big\|_{\infty}\\
    \leq&\ \big\||g_{k}|e_{4}\big\|_{\infty}+\big\||h^{\ast}_{k}|e_{4}\big\|_{\infty}\\
    =&\ \big\|e_{4}|g_{k}|^{2}e_{4}\big\|^{\frac{1}{2}}_{\infty}
    +\big\|e_{4}|h^{\ast}_{k}|^{2}e_{4}\big\|^{\frac{1}{2}}_{\infty}\\
   \leq&\ \big\|e_{4}\sum_{k}|g_{k}|^{2}e_{4}\big\|^{\frac{1}{2}}_{\infty}
   +\big\|e_{4}\sum_{k}|h^{\ast}_{k}|^{2}e_{4}\big\|^{\frac{1}{2}}_{\infty}\\
    =&\ \big\|(\sum_{k}|g_{k}|^{2})^{\frac{1}{2}}e_{2}e_{4}\big\|_{\infty}
    +\big\|(\sum_{k}|h^{\ast}_{k}|^{2})^{\frac{1}{2}}e_{3}e_{4}\big\|_{\infty}\\
    \leq&\ 2\lambda,
\end{align*}
where in the second equality we used the polar decomposition.

Finally, if we set $q=e_{1}\bigwedge e_{4}$,  then it is easy to see that this is the desired projection. Thus, we finish the proof of Corollary \ref{cor2}.
\end{proof}

\section{$(L_{\infty},\mathrm{BMO})$ estimate}
In this section, we examine the $(L_{\infty},\mathrm{BMO})$ estimate. The dyadic $\mathrm{BMO}$ spaces $\mathrm{BMO}_{d}(\mathcal A)$ were defined in the Introduction.

In order to prove conclusion (ii), it suffices to show
\begin{align}\label{94}
\Big\|
\sum_{k} T_k \hskip-1pt f \otimes e_{k1}
\Big\|_{\mathrm{BMO}_{d}(\mathcal{A})} \lesssim \,
\|f\|_\infty.
\end{align}
Indeed, (\ref{94}) is equivalent to
\begin{align}\label{95}
\Big\|
\sum_{k} T_k \hskip-1pt f \otimes e_{k1}
\Big\|_{\mathrm{BMO}_{d}^{c}(\mathcal A)} \lesssim \,
\|f\|_\infty
\end{align}
and
\begin{align}\label{96}
\Big\|
\sum_{k} T_k \hskip-1pt f \otimes e_{k1}
\Big\|_{\mathrm{BMO}_{d}^{r}(\mathcal A)} \lesssim \,
\|f\|_\infty.
\end{align}
Using the fact $\|g\|_{\mathrm{BMO}_{d}^{c}(\mathcal{A})}=\|g^{\ast}\|_{\mathrm{BMO}_{d}^{r}(\mathcal{A})}$ and taking the adjoint of both sides in (\ref{95}), we have
\begin{align*}
\Big\|
\sum_{k} T_k \hskip-1pt f \otimes e_{1k}
\big\|_{\mathrm{BMO}_{d}^{r}(\mathcal{A})}
& =\Big\|
\big(\sum_{k}T_k \hskip-1pt f \otimes e_{1k}
\big)^{\ast}\Big\|_{\mathrm{BMO}_{d}^{c}(\mathcal{A})}\\
& =\Big\|
\sum_{k} T_k \hskip-1pt f^{\ast} \otimes e_{k1}
\Big\|_{\mathrm{BMO}_{d}^{c}(\mathcal{A})}\lesssim \,
\|f^{\ast}\|_\infty=\|f\|_\infty.
\end{align*}
Similarly, we use (\ref{96}) to get
$$\Big\|
\sum_{k} T_k \hskip-1pt f \otimes e_{1k}\Big\|_{\mathrm{BMO}_{d}^{c}(\mathcal{A})} \lesssim \, \|f\|_\infty.$$
These imply
$$\Big\|
\sum_{k} T_k \hskip-1pt f \otimes e_{1k}\Big\|_{\mathrm{BMO}_{d}(\mathcal{A})} \lesssim \, \|f\|_\infty.$$

We will use the fact that in the definition of the $\mathrm{BMO}$ norm of a function
$f$, the average $f_Q$ can be replaced by any operator $\alpha_Q$ depending on $Q$ (see e.g. Lemma 1.3 of \cite{JMP}).

Now we are ready to prove the second part of Theorem \ref{t5}.
\begin{proof}
For $f\in L_{\infty}(\mathcal N)$, and a dyadic cube $Q$, we decompose $f$ as
 $f=f\1_{ 3Q}+f\1_{\R^d\setminus 3Q}\triangleq f_1+f_2$.
We shall take $\alpha_{Q,k}=T_{k}f_2(c_Q)$ and $\alpha_{Q}=\sum\limits_{k}\alpha_{Q,k}\otimes e_{k1}$, where $c_Q$ is the center of $Q$. Write $T_{k}f-\alpha_{Q,k}$ as
$$T_{k}f-\alpha_{Q,k}=(T_{k}f-T_{k}f_2)+(T_{k}f_2-\alpha_{Q,k})\triangleq B_{k1}f+B_{k2}f.$$
We first prove (\ref{95}). Using the operator convexity of square function $x\mapsto|x|^{2}$, we obtain
$$\big|\sum_{k}(T_{k}f-\alpha_{Q,k})\otimes e_{k1}\big|^{2}\leq2\big|\sum_{k}B_{k1}f\otimes e_{k1}\big|^{2}+2\big|\sum_{k}B_{k2}f\otimes e_{k1}\big|^{2}.$$
The first term $B_{1}f=\sum\limits_{k} B_{k1} \hskip-1pt f \otimes e_{k1}$ is easy to estimate. Indeed,
\begin{align*}
&\  \Big\| \big(\frac{1}{|Q|}\int_Q (B_1
f(x))^\ast (B_1f(x)) \, dx\big)^{\frac{1}{2}} \Big\|_{\M \bar\otimes
\mathcal{B}(\ell_{2})}^{2}\\
&\ =\Big\| \big(\frac{1}{|Q|}\int_Q \sum_{k}|T_{k}f_{1}(x)|^{2} \, dx\big)^{\frac{1}{2}} \Big\|^{2}_{\M}\\
& =\frac{1}{|Q|}\Big\| \int_Q \sum_{k}|T_{k}f_{1}(x)|^{2} \, dx \Big\|_{\M}\\
& = \frac{1}{|Q|} \sup_{\|a\|_{L_2(\M)} \le 1}\tau \otimes \int_Q
\sum_{k}|T_{k}f_{1}a(x)|^{2}dx\\
& \le \frac{1}{|Q|} \sup_{\|a\|_{L_2(\M)} \le 1}\tau \otimes \int_{\R^{d}}
\sum_{k}|T_{k}f_{1}a(x)|^{2}dx\\
& = \frac{1}{|Q|} \sup_{\|a\|_{L_2(\M)} \le 1} \|T(fa
1_{3Q})\|_{L_2(\mathcal{N};\ell_{2}^{rc})}^2\\
& \lesssim \frac{1}{|Q|} \sup_{\|a\|_{L_2(\M)} \le 1} \|fa
1_{3Q}\|_{2}^2\lesssim \|f\|_\infty^{2},
\end{align*}
where in the third equality, we considered elements in $\M$ as bounded linear operators on $L_{2}(\M)$ by left multiplication and in the last inequality we used the $L_2$-boundedness of $T$---Lemma \ref{t1}.

Now we turn to the second term $B_{2}f=\sum\limits_{k} B_{k2} \hskip-1pt f \otimes e_{k1}$. We have
\begin{align*}
B_{2}f(x)^{\ast}B_{2}f(x)
&\ =\sum_{k}|T_{k}f_{2}(x)-T_{k}f_{2}(c_{Q})|^{2}\\
&\ =\sum_{k}|(M_{k}f_{2}(x)-\mathsf{E}_{k}f_{2}(x))-
(M_{k}f_{2}(c_{Q})-\mathsf{E}_{k}f_{2}(c_{Q}))|^2\\
&\ \triangleq\sum_{k}|F_{k,Q}(x)|^{2}.
\end{align*}
We first claim that for any $k$ satisfying $2^{-k}<\ell(Q)$, we have for any $x\in Q$, $F_{k,Q}(x)=0.$ Indeed when $2^{-k}<\ell(Q)$, a simple geometric observation implies that both $\mathsf{E}_k f_2$ and $M_{k}f_2$ are supported in $\R^d\setminus Q$ since $f_2$ is supported in $\R^d\setminus 3Q$. For $k$ satisfying $2^{-k}\geq\ell(Q)$, we claim that for any $x\in Q$,
\begin{eqnarray}\label{97}
\|F_{k,Q}(x)\|_{\M}\lesssim2^{k}\ell(Q)\|f\|_{\infty}.
\end{eqnarray}
Indeed, note that $x$ and $c_Q$ are in the same cube in $\Q_{k}$, so $\mathsf{E}_kf_2(x)=\mathsf{E}_k f_2(c_Q)$. On the other hand,
\begin{align*}
\|M_{k} f_2(x)-M_{k} f_2(c_Q)\|_{\M} &=\frac{1}{|B_{k}|} \Big\|\int_{B_{k}+x}f_2(y)dy -\int_{B_{k}+ c_Q}f_2(y)dy\Big\|_{\M} \\
&=2^{kd} \Big\|\int_{\R^d}f_2(y)( \1_{ (B_{k}+ c_Q)\setminus  (B_{k}+x)}-\1_{(B_{k} +x)\setminus (B_{k}+c_Q)})(y)dy\Big\|_{\M} \\
&\leq2^{kd}|  (B_{k}+ c_Q) \Delta  (B_{k}+x) |\cdot \|f\| _{\infty}.
\end{align*}
Then the fact that $|(B_{k}+ c_Q) \Delta (B_{k}+x) |\leq C_{d}2^{-k(d-1)}|x-c_Q|$ yields
$$\|M_{k} f_2(x)-M_{k} f_2(c_Q)\|_{\M}\leq C_{d}2^{k}|x-c_Q|\|f\|_{\infty}\lesssim 2^{k}\ell(Q)\|f\|_{\infty}.$$
This is precisely claim (\ref{97}).
Putting all these observations together, we obtain
\begin{align*}
 &\  \Big\| \Big(\frac{1}{|Q|}\int_Q (B_2
f(x))^\ast (B_2f(x)) \, dx\Big)^{\frac{1}{2}} \Big\|_{\M \bar\otimes
\mathcal{B}(\ell_{2})}^{2}\\
 &\ =\Big\| \Big(\frac{1}{|Q|}\int_Q |Tf_2(x)-\alpha_Q\Big|^2\, dx\Big)^{\frac{1}{2}} \Big\|^{2}_{\M}\\
&=\frac{1}{|Q|}\Big\|\int_{Q}\sum_{k}|F_{k,Q}(x)|^{2}\, dx\Big\|_{\M}\\
&\leq\frac{1}{|Q|}\int_{Q}\sum_{k}\|F_{k,Q}(x)\|_{\M}^{2}\, dx\\
&\lesssim \ell(Q)^{2}\cdot  \|f\|^{2}_{\infty}\sum_{k:2^{-k} \geq\ell(Q)}2^{2k}\lesssim \|f\|^{2}_{\infty}.
\end{align*}
This yields the desired estimate.

We now consider (\ref{96}). Here we only need to deal with $B_{1}f=\sum\limits_{k} B_{k1} \hskip-1pt f \otimes e_{k1}$, while $B_{2}f$ can be treated as before.
We note that
\begin{align*}
& \Big\| \Big(\frac{1}{|Q|}\int_Q (B_1
f(x)) (B_1f(x))^\ast \, dx\Big)^{\frac{1}{2}} \Big\|_{\M \bar\otimes
\mathcal{B}(\ell_{2})}^{2}\\
& = \frac{1}{|Q|} \Big\| \int_Q (B_1 \hskip-1pt
f(x))(B_1 \hskip-1pt f(x))^\ast \, dx \Big\|_{\M \bar\otimes
\mathcal{B}(\ell_{2})}\\
& = \frac{1}{|Q|} \Big\|
\sum_{k_1,k_2} \Big[ \int_Q T_{k_1}f_{1}(x)
T_{k_2} f_{1}^\ast(x) \, dx \Big] \otimes e_{k_1,k_2} \Big\|_{\M
\bar\otimes \mathcal{B}(\ell_{2})} \\
&\triangleq \frac{1}{|Q|}
\|\Lambda\|_{\M \bar\otimes \mathcal{B}(\ell_{2})}.
\end{align*}
Since $\Lambda$ is a positive operator acting on $\ell_{2}(L_2(\M))$ $(=L_2(\M;\ell_{2}^{rc}))$, we have
\begin{align*}
&\ \frac{1}{|Q|} \Big\| \int_Q (B_1 \hskip-1pt
f(x))(B_1 \hskip-1pt f(x))^\ast \, dx \Big\|_{\M \bar\otimes
\mathcal{B}(\ell_{2})}\\  & = \frac{1}{|Q|} \sup_{\|a\|_{L_2(\M;\ell_{2}^{rc})} \le 1} \big\langle \Lambda a, a \big\rangle \\
& =  \frac{1}{|Q|} \sup_{\|a\|_{L_2(\M;\ell_{2}^{rc})} \le 1} \hskip4.7pt
\tau \Big( \big[ \sum_{k_1} a_{k_1}^\ast \otimes e_{1k_1}
\big] \, \Lambda \,
\big[ \sum_{k_2} a_{k_2} \otimes e_{k_21} \big] \Big) \\
& = \frac{1}{|Q|} \sup_{\|a\|_{L_2(\M;\ell_{2}^{rc})} \le 1}
\hskip4.7pt \tau \otimes \int_Q \Big| \sum_{k}
T_k f_{1}^\ast(x) a_k \Big|^2 \, dx \\
& \leq \frac{1}{|Q|} \sup_{\|a\|_{L_2(\M;\ell_{2}^{rc})} \le 1}
\hskip4.7pt \tau \otimes \int_{\R^d} \Big| \sum_{k}
T_k (f_{1}^\ast a_k )(x) \Big|^2 \, dx \\
& \lesssim  \frac{1}{|Q|} \sup_{\|a\|_{L_2(\M;\ell_{2}^{rc})} \le 1}
\hskip4.7pt \tau \otimes \int_{\R^d}\sum_{k}
|f_{1}^\ast a_k(x)|^2 \, dx \\
& \le  \frac{1}{|Q|}
\sup_{\|a\|_{L_2(\M;\ell_{2}^{rc})} \le 1}
\hskip4.7pt \tau (\sum_{k}|a_{k}|^2) |3Q| \, \|f\|_\infty^2 \ \lesssim
\ \|f\|_\infty^2,
\end{align*}
where in the second inequality we used the $L_{2}(\mathcal{N})$-boundedness of the adjoint of $T$. This proves (\ref{96}). Therefore, the estimates obtained so far together with their row analogues give rise to
$$\max \Big\{ \Big\|
\sum_{k} T_k \hskip-1pt f \otimes e_{k1}
\Big\|_{\mathrm{BMO}_{d}(\mathcal{A})}, \Big\|
\sum_{k} T_k \hskip-1pt f \otimes e_{1k}
\Big\|_{\mathrm{BMO}_{d}(\mathcal{A})} \Big\} \lesssim \|f\|_\infty.$$
This completes the
BMO estimate.
\end{proof}

\section{Strong type $(p,p)$ estimates}
In this section, we show the strong type $(p,p)$ estimate of $(T_k)$ for $1<p<\infty$. 
\begin{prop}\label{11}
Let $1<p<\infty$. Then $(T_k)$ is bounded from $L_{p}(\mathcal{N})$ to $L_p(\mathcal{N};
\ell_{2}^{rc})$.
\end{prop}

\begin{proof}
The result for $p=2$ is just Lemma \ref{t1}. For the case of $1<p<2$,
using the weak tyep $(1,1)$ estimate of ${T}$ obtained in Section 3 and Lemma \ref{t1}, we conclude that ${T}$ is bounded from $L_{p}(\mathcal{N})$ to $L_p(L_{\infty}(\Omega)\overline{\otimes}\mathcal{N})$ by real interpolation. Thus $(T_k)$ is bounded from $L_{p}(\mathcal{N})$ to $L_p(\mathcal{N};
\ell_{2}^{rc})$ thanks to noncommutative Khintchine's inequalities \cite{Pis09, PiRi17}.

We now turn to the case of $2<p<\infty$.
If we set $T_{c}f=\sum\limits_{k} T_k \hskip-1pt f \otimes e_{k1}$ and $T_{r}f=\sum\limits_{k} T_k \hskip-1pt f \otimes e_{1k}$, then Lemma \ref{t1} in conjunction with $\mathrm{BMO}$ type estimate yields that $T_{c}$ and $T_{r}$ are bounded from $L_p(\mathcal N)$ to $L_p(\mathcal{N}\bar\otimes\mathcal B({\ell_2}))$ by complex interpolation \cite{Mu}. Therefore, $(T_k)$ is bounded from $L_p(\mathcal N)$ to $L_p(\mathcal{N};
\ell_{2}^{rc})$ for all $2\leq p<\infty$ by noncommutative Khintchine's inequalities.
\end{proof}

\section*{Appendix. Proof of Lemma \ref{t1}---$L_2$-boundedness }
As mentioned in the introduction, the noncommutative $L_2$-boundedness of the square function follows trivially from the corresponding commutative result. However, in this appendix, we provide a proof using similar idea as presented in \cite[Theorem 2.3]{HM1} but not using the commutative result as a black box. It is this proof that inspires us to complete the whole paper.

\begin{proof}[Proof of Lemma \ref{t1}.]
Let $h\in L_{2}(\mathcal N)$. Without loss of generality,
we can assume that $h$  is positive.
Let $h=\sum\limits_{n\in\Z}dh_{n}$, since martingale differences are orthogonal, we have
$$\sum_{n\in\Z}\|dh_{n}\|_{2}^{2}=\|h\|_{2}^{2}.$$
On the other hand,
$$\|Th\|_{L_2(L_{\infty}(\Omega)\overline{\otimes}\mathcal{N})}^{2}
=\sum_{k}\|(M_{k}-\mathsf{E}_{k})\sum_{n}dh_{n}\|_{2}^{2}.$$
To this end, according to the almost orthogonality principle---Lemma \ref{L1}, it suffices to prove
\begin{align}\label{3}
\|(M_{k}-\mathsf{E}_k)dh_{n} \|^{2}_{2}\lesssim 2^{-|n-k|}\|dh_{n} \|^{2}_{2}.
\end{align}
We first prove (\ref{3}) in the case $k\geq n$. Note that $\mathsf{E}_kdh_{n}=dh_{n}$ for $k\geq n$. It is enough to show
\begin{align}\label{4}
\|M_{k}dh_{n}-dh_n\|^{2}_{2}\lesssim 2^{n-k}\|dh_n\|^{2}_{2}.
\end{align}
We write
\begin{align*}
    \|M_{k}dh_n-dh_n\|^{2}_{2}
    =&\ \int_{\R^d}\|M_{k}dh_n(x)-dh_n(x)\|^{2}_{L_{2}(\M)}dx\\
    =&\ \sum_{H\in\Q_{n}}\int_H\|M_{k}dh_n(x)-dh_n(x)\|^{2}_{L_{2}(\M)}dx.
\end{align*}
Since $dh_{n}$ is a constant operator on $H\in\Q_{n}$, we have $M_{k}dh_n(x)-dh_n(x)=0$ if $x+ B_{k}\subset H$. Thus, for $x\in H$, $(M_{k}dh_n-dh_n)(x)$ may be nonzero only if $x+B_{k}$ intersects with the complement of $H$---$H^c$. Hence, for a given atom $H$ and a given Euclidean ball $B$ in $\R^{d}$, we define
  $$\mathcal{H}(B,H)= \{x\in H|x+B\cap H^c\neq \emptyset\}.$$
For a fixed $H\in\Q_{n}$, it is easy to observe that
  $$|\mathcal{H} (B_{k},H)|\lesssim2^{(d-1)(-n)}\cdot 2^{-k}.$$
Let $m_H$ be the maximum of $\|dh_n(x)\|_{L_{2}(\M)}$ on $H$ and the cubes in $\Q_{n}$ neighboring $H$. Then $\|M_{k}dh_n(x)-dh_n(x)\|_{L_{2}(\M)} \leq 2m_H$ for every $x\in H$. Therefore
\begin{align}\label{10}
 \int_H \|M_{k}dh_n(x)-dh_n(x)\|^{2}_{L_{2}(\M)}dx\lesssim2^{(d-1)(-n)}\cdot2^{-k}\cdot m_H^{2}\lesssim2^{n-k} \int_Hm_H^{2}dx.
\end{align}
Since $m_H$ is a constant on $H$, then $\int_{\R^d}m^{2}_H \leq  C_d\int_{\R^d}\|dh_n(x)\|^{2}_{L_{2}(\M)}dx$. Thus, summing over all $H\in\Q_{n}$ in (\ref{10}), we obtain the desired estimate (\ref{4}).

Now we handle  (\ref{3}) in the case $n> k$. Note that $\mathsf{E}_kdh_n=0$ in this case. Hence, it suffices to prove
\begin{align}\label{5}
\|M_{k}dh_n\|^{2}_2\lesssim2^{k-n}\|dh_n\|^{2}_{2}.
\end{align}
Note that by the cancellation property of $h_n$ over atoms in $\Q_{n-1}$, we see that $M_{k}dh_n=M_{k,n-1}dh_n$. Hence, we use Lemma \ref{kn} to obtain (\ref{5}). Therefore, the proof of Lemma \ref{t1} is complete.
\end{proof}

\end{document}